\documentclass[12pt]{amsart}

\usepackage[usenames,dvipsnames]{xcolor}
\usepackage{amssymb,amscd,amsmath,hyperref,color,enumerate}
\usepackage{mathrsfs}

\usepackage[margin=3cm]{geometry}

\newcommand{\N}{{\ensuremath{\mathbb{N}}}}

\newcommand{\C}{{\ensuremath{\mathbb{C}}}}

\usepackage[normalem]{ulem}

\newcommand{\stkout}[1]{\ifmmode\text{\sout{\ensuremath{#1}}}\else\sout{#1}\fi}

\DeclareMathOperator{\supp}{supp}

\DeclareMathOperator{\diag}{diag}

\DeclareMathOperator{\spn}{span}

\newcommand{\symm}{\ensuremath{\odot}}

\newcommand{\abs}[1]{\ensuremath{ {\left| #1 \right|} }}

\newcommand{\norm}[1]{\ensuremath{ {\left\| #1 \right\|} }}

\numberwithin{equation}{section}

\newtheorem{proposition}{Proposition}[section]
\newtheorem{lemma}[proposition]{Lemma}

\newtheorem{theorem}[proposition]{Theorem}
\newtheorem{corollary}[proposition]{Corollary}
\theoremstyle{definition}

\newtheorem{problem}[proposition]{Problem}

\newtheorem{example}[proposition]{Example}

\newtheorem{remark}[proposition]{Remark}

\newtheorem{step}{Step}
\newtheorem{stepa}{Step}[step]

\numberwithin{equation}{section}

\newlength{\leftstackrelawd}
\newlength{\leftstackrelbwd}
\def\leftstackrel#1#2{\settowidth{\leftstackrelawd}%
	{${{}^{#1}}$}\settowidth{\leftstackrelbwd}{$#2$}%
	\addtolength{\leftstackrelawd}{-\leftstackrelbwd}%
	\leavevmode\ifthenelse{\lengthtest{\leftstackrelawd>0pt}}%
	{\kern-.5\leftstackrelawd}{}\mathrel{\mathop{#2}\limits^{#1}}}

\begin{document}
	
	\title[Jordan embeddings between block upper-triangular subalgebras]{Characterizing Jordan embeddings between block upper-triangular subalgebras via preserving properties}
	
	\author{Ilja Gogi\'{c}, Tatjana Petek, Mateo Toma\v{s}evi\'{c}}
	
	\address{I.~Gogi\'c, Department of Mathematics, Faculty of Science, University of Zagreb, Bijeni\v{c}ka 30, 10000 Zagreb, Croatia}
	\email{ilja@math.hr}

	\address{T.~Petek, Faculty of Electrical Engineering and Computer Science, University of Maribor, Koro\v{s}ka cesta 46, SI-2000 Maribor, Slovenia; \newline
		Institute of Mathematics, Physics and Mechanics, Jadranska ulica 19, SI-1000 Ljubljana, Slovenia }
	\email{tatjana.petek@um.si}
	
	\address{M.~Toma\v{s}evi\'c, Department of Mathematics, Faculty of Science, University of Zagreb, Bijeni\v{c}ka 30, 10000 Zagreb, Croatia}
	\email{mateo.tomasevic@math.hr}
	
	\thanks{We thank the referee for useful suggestions that helped us to improve the presentation of our results.}


	
	\date{\today}

	\begin{abstract}
		Let $M_n$ be the algebra of $n \times n$ complex matrices. We consider arbitrary subalgebras $\mathcal{A}$ of $M_n$ which contain the algebra of all upper-triangular matrices (i.e.\ block upper-triangular
		subalgebras), and their Jordan embeddings. We first describe Jordan embeddings $\phi : \mathcal{A} \to M_n$ as maps of the form $\phi(X)=TXT^{-1}$ or $\phi(X)=TX^tT^{-1}$, where $T\in M_n$ is an invertible matrix, and then we obtain a simple criteria of when one block upper-triangular subalgebra  Jordan-embeds into another (and in that case we describe the form of such embeddings). As a main result, we characterize Jordan embeddings $\phi : \mathcal{A} \to M_n$ (when $n\geq 3$) as continuous injective maps which preserve commutativity and spectrum. We show by counterexamples that all these assumptions are indispensable (unless $\mathcal{A} = M_n$ when injectivity is superfluous).
	\end{abstract}
	
	\maketitle
	
	\setlength{\parindent}{18pt}
	
	\section{Introduction}

	Jordan algebras were first introduced by Pascual Jordan in 1933 in the context of quantum mechanics \cite{Jordan}. The majority of the practically relevant Jordan algebras naturally arise as subalgebras of an associative algebra $\mathcal{A}$ under a symmetric product given by $$x \circ y = xy + yx.$$
	This gives rise to the study of Jordan homomorphisms in the context of associative rings and algebras. Namely, recall that an additive (linear) map $\phi : \mathcal{A} \to \mathcal{B}$ between rings (algebras) $\mathcal{A}$ and $\mathcal{B}$ is a \emph{Jordan homomorphism} if
	
	$$\phi(a \circ b) = \phi(a) \circ \phi(b), \qquad \text{ for all }a,b \in \mathcal{A}.$$
	When the rings (algebras) are $2$-torsion-free, this is equivalent to
	$$\phi(a^2) = \phi(a)^2, \qquad \text{ for all }a \in \mathcal{A}.$$
	Clearly, multiplicative and antimultiplicative maps are immediate examples of such maps. One of the main problems in the context of Jordan homomorphisms is under which assumptions on rings (algebras) $\mathcal{A}$ and  $\mathcal{B}$, usually without $2$-torsion, can we conclude that every Jordan homomorphism $\phi : \mathcal{A} \to \mathcal{B}$ (possibly satisfying some extra conditions such as surjectivity) is either multiplicative or antimultiplicative. More generally, the question is whether one can express all such Jordan homomorphisms as a suitable combination of ring (algebra) homomorphisms and antihomomorphisms. This question goes a long way back. Namely, in 1950 Jacobson and Rickart \cite{JacobsonRickart} proved that a Jordan homomorphism from an arbitrary ring into an integral domain is either a homomorphism or an antihomomorphism. This paper is particularly relevant for our discussion as it also proves that a Jordan homomorphism of the ring of $n\times n$ matrices, $n \ge 2$, over an arbitrary unital ring is the sum of a homomorphism and an antihomomorphism. In the same vein, Herstein \cite{Herstein} showed that a Jordan homomorphism onto a prime ring
	is either a homomorphism or an antihomomorphism. The same result was later refined by Smiley \cite{Smiley}.
	
	Let $M_n$ be the algebra of $n \times n$ matrices over the field of complex numbers. By combining the aforementioned result of Herstein with the well-known fact that all automorphisms of $M_n$ are inner (for a short and elegant proof see \cite{Semrl2}), one obtains that all nonzero Jordan endomorphisms $\phi$ of $M_n$ are precisely maps of the form
	\begin{equation}\label{eq:inner}
		\phi(X) = TXT^{-1} \qquad \text{or}   \qquad \phi(X) = TX^{t}T^{-1}
	\end{equation}
	(globally) for some invertible matrix $T \in M_n$.
	
	There have been many attempts to characterize Jordan homomorphisms, particularly on matrix algebras, using preserver properties. These attempts date back at least to 1970 and Kaplansky's famous problem \cite{Kaplansky} which asks under which conditions on unital (complex) Banach algebras $A$ and $B$ is a linear unital map $\phi : A \to B$ which preserves invertibility necessarily a Jordan homomorphism. This problem received a lot of attention and progress was made in some special cases, but it is still widely open, even for $C^*$-algebras (see \cite{BresarSemrl}, page 270). For other interesting types of linear preserver problems resulting in more general kinds of maps, we refer to the survey paper \cite{LiPierce} and references within. We would also like to distinguish the following nonlinear preserver problem which elegantly characterizes Jordan automorphisms of $M_n$:
	
	\begin{theorem}[\v{S}emrl]\label{thm:Semrl}
		Let $\phi : M_n \to M_n, n \ge 3$ be a continuous map which preserves commutativity and spectrum. Then there exists an invertible matrix $T \in M_n$ such that $\phi$ is of the form \eqref{eq:inner}.
	\end{theorem}
	
	A precursor to this result was first formulated in \cite{PetekSemrl} and it assumed its current form a decade later in \cite{Semrl}. It also serves as the main motivation for our investigation. Namely, we are interested in the following general problem:
	\begin{problem}\label{Prob:main}
		Find necessary and sufficient conditions on a (Jordan) subalgebra $\mathcal{A}$ of $M_n$ such that each Jordan automorphism of $\mathcal{A}$, or more generally a Jordan embedding (monomorphism) $\phi : \mathcal{A} \to M_n$, extends to a Jordan automorphism of $M_n$. Additionally, for such $\mathcal{A}$, characterize all such maps $\phi$ via suitable preserving properties, similarly as in Theorem \ref{thm:Semrl}.
	\end{problem}
	
	The first natural example to consider in the context of Problem \ref{Prob:main} is the algebra $\mathcal{T}_n$ of $n\times n$ upper-triangular complex matrices. First of all, it is well-known that all Jordan automorphisms of $\mathcal{T}_n$ are of the form \eqref{eq:inner} for a suitable invertible matrix $T \in M_n$ (see e.g.\ \cite[Corollary 4]{MolnarSemrl}). The same holds true for all Jordan embeddings $\mathcal{T}_n \to M_n$ (a special case of our first result, Theorem \ref{thm:Jordan monomorphisms on A}). Also, Jordan automorphisms of $\mathcal{T}_n$ (as well as more general type of maps on $\mathcal{T}_n$) were characterized via both linear and nonlinear preserving properties by several authors (see e.g.\ \cite{CheungLi, ChooiLim, ChoiLimSemrl, CuiHouLi, Huang, LiSemrlSoares, Petek, ZhaoWangWang}).  In particular, following Theorem \ref{thm:Semrl}, the second-named author \cite{Petek} described all Jordan automorphisms of $\mathcal{T}_n$ as continuous spectrum and commutativity preserving surjective maps $\mathcal{T}_n \to \mathcal{T}_n$. Analyzing the main result of \cite{Petek} it is easy to verify that the same holds true if instead of surjectivity one assumes injectivity, thus providing a positive solution for (both parts of) Problem \ref{Prob:main} when $\mathcal{A}=\mathcal{T}_n$.
	
	Continuing in this vein, the next class of algebras we consider are subalgebras $\mathcal{A}$ of $M_n$ which contain $\mathcal{T}_n$. It turns out that these algebras are precisely the \emph{block upper-triangular subalgebras} of $M_n$ (see \cite{WangPanWang}). More specifically, any such algebra is of the form
	\begin{equation}\label{eq:definition of A}
		\mathcal{A}_{k_1, \ldots, k_r}:=\begin{bmatrix} M_{k_1,k_1} & M_{k_1,k_2} & \cdots & M_{k_1,k_r} \\ 0 & M_{k_2,k_2} & \cdots & M_{k_2,k_r} \\
			\vdots & \vdots & \ddots & \vdots \\
			0 & 0 & \cdots & M_{k_r,k_r} \end{bmatrix}
	\end{equation}
	for some $r, k_1, \ldots, k_r \in \N$ such that $k_1+\cdots+k_r = n$. These algebras also appear in the literature as parabolic algebras (see e.g.\ \cite{Agore, WangPanWang}). Note that the block upper-triangular algebras $\mathcal{A}_{1, n-1}$ and $\mathcal{A}_{n-1, 1}$ are exactly (up to similarity) the unital strict subalgebras of $M_n$ of maximal dimension (see \cite{Agore}). Our first introductory result verifies that block upper-triangular algebras indeed satisfy the desired extension property, providing an affirmative answer for the first part of Problem \ref{Prob:main}:
	
	\begin{theorem}\label{thm:Jordan monomorphisms on A}
		Let $\mathcal{A} \subseteq M_n$ be a block upper-triangular subalgebra and let $\phi : \mathcal{A} \to M_n$ be a Jordan embedding. Then there exists an invertible matrix $T \in M_n$ such that $\phi$ is of the form \eqref{eq:inner}.
	\end{theorem}
	After providing the short proof of Theorem \ref{thm:Jordan monomorphisms on A}  (in Section \S\ref{sec:block upper-triangular subalgebras and their Jordan embeddings}) we also present simple criteria of when one block upper-triangular algebra (Jordan-)embeds into another (see Corollary \ref{cor:Jordan monomorphism unified} and \ref{cor:when are block upper-triangular algebras isomorphic}). To put Theorem \ref{thm:Jordan monomorphisms on A} in a wider context, there is a number of related (and somewhat more general) results concerning Jordan homomorphisms of certain classes of (block) upper-triangular rings and algebras. For example, an influential result by Beidar, Bre\v{s}ar and Chebotar states that every Jordan isomorphism of the algebra of upper-triangular matrices $\mathcal{T}_n(\mathcal{C}), n \ge 2$ over a $2$-torsion-free commutative unital ring $\mathcal{C}$ without nontrivial idempotents onto an arbitrary $\mathcal{C}$-algebra is necessarily multiplicative or antimultiplicative \cite{BeidarBresarChebotar}. A generalization was given in \cite{LiuTsai} by removing the assumption that $\mathcal{C}$ has no nontrivial idempotents. These results were further developed by Benkovi\v c in \cite{Benkovic}; any Jordan homomorphism from $\mathcal{T}_n(\mathcal{C}), n \ge 2$ into an algebra $\mathcal{B}$ is a (so-called) near-sum of a homomorphism and an antihomomorphism. We also mention papers \cite{Boboc, DuWang, WangWang, Wong} which treat similar problems for Jordan homomorphisms between more general types of (block) triangular matrix rings.
	
	Concerning block upper-triangular subalgebras in particular, the papers \cite{Akkurt, Akkurt2, BrusamarelloFornaroliKhrypchenko1, BrusamarelloFornaroliKhrypchenko2} (in fact, \cite{Akkurt, Akkurt2, BrusamarelloFornaroliKhrypchenko1, BrusamarelloFornaroliKhrypchenko2} are sequels of the aforementioned paper \cite{Benkovic}) describe Jordan homomorphisms of certain classes of more general algebras. As block upper-triangular algebras are (up to isomorphism) precisely finite-dimensional instances of nest algebras, papers \cite{ChenLuChen, Lu} are also relevant. Note that many of the mentioned results above assume that the image of Jordan homomorphisms are rings or algebras. In Theorem \ref{thm:Jordan monomorphisms on A} we make no such assumption but this is obviously compensated by the fact that the codomain is restricted to $M_n$, which incidentally also enables us to state the explicit form \eqref{eq:inner} of such maps. 
	
	After verifying that block upper-triangular algebras $\mathcal{A}$ satisfy the first part of Problem \ref{Prob:main}, the next step is to characterize Jordan embeddings $\phi : \mathcal{A} \to M_n$ via suitable preserver properties. Building upon both Theorem \ref{thm:Semrl} and \cite[Corollary 3]{Petek}, we arrived at the following theorem, which is also the main result of our paper:
	
	\begin{theorem}\label{thm:main result}
		Let $\mathcal{A} \subseteq M_n, n \ge 3$, be a block upper-triangular subalgebra and let $\phi : \mathcal{A} \to M_n$ be a continuous injective map which preserves commutativity and spectrum. Then $\phi$ is a Jordan embedding and hence of the form \eqref{eq:inner} for some invertible matrix $T \in M_n$.
	\end{theorem}
	Moreover, if we additionally assume that the image of $\phi$ is contained in $\mathcal{A}$, so that $\phi : \mathcal{A} \to \mathcal{A}$, using the invariance of domain theorem we show that the spectrum preserving assumption can be further relaxed to spectrum shrinking (Corollary \ref{cor:invdomain}).
	
	\bigskip
	
	This paper is organized as follows. We begin by providing terminology and notation in Section \S\ref{sec:prel} along with some preliminary technical results related to block upper-triangular algebras and Jordan homomorphisms. Section \S\ref{sec:block upper-triangular subalgebras and their Jordan embeddings} contains the proof of Theorem \ref{thm:Jordan monomorphisms on A} and its consequences regarding (Jordan) embeddings between two block upper-triangular subalgebras of $M_n$. Section \S\ref{sec:Proof of the main result} is the core of the paper. It contains the proof of Theorem \ref{thm:main result}, which is nontrivial and contains most of the actual work. The proof is carried out by induction on $n$. After proving Theorem \ref{thm:main result} a few direct consequences are stated. Finally, in Section \S\ref{sec:Counterexamples}, we demonstrate the necessity of all assumptions of Theorem \ref{thm:main result} via counterexamples.
	
	\setlength{\parindent}{0pt}
	
	\section{Preliminaries}\label{sec:prel}
	We begin this section by introducing some notation and terminology. First of all, for a unital algebra  $\mathcal{A}$ by $\mathcal{A}^{\times}$ we denote the set of all invertible elements in $\mathcal{A}$.
	
	\smallskip
	
	Let $n \in\N$.
	\begin{itemize}
		\item As usual, by $M_n := M_n(\C)$ we denote the set of all $n\times n$ complex matrices and by $M_{m,n}$ ($m \in \N$) the set of all $m\times n$ complex matrices.
		\item For $A,B \in M_n$ we denote by $A \leftrightarrow B$ the fact that $A$ and $B$ commute, i.e.\ $AB = BA$.
		\item We say that the matrices $A,B \in M_n$ are orthogonal, and write $A \perp B$, if $AB=BA=0$. Furthermore, by $A^\perp$ we denote the set of all matrices in $M_n$ orthogonal to $A$.
		\item For $A,B \in M_n$ we write $A\circ B=AB+BA$.
		\item For $A \in M_n$ we denote the characteristic polynomial of $A$ by $p_A(x) = \det(A - xI)$, by $R(A)$ the image of $A$ and by $N(A)$ the nullspace of $A$.
		\item By $\mathcal{T}_n$ and $\mathcal{D}_n$ we denote the sets of all upper-triangular and diagonal matrices of $M_n$, respectively.
		\item By $\mathcal{A}_{k_1, \ldots, k_r}$ we denote the block upper-triangular algebra \eqref{eq:definition of A} for some $r$, $k_1, \ldots, k_r \in \N$ such that $k_1+\cdots+k_r = n$.

		\item By $\Delta$ we denote the diagonal matrix  $\diag(1,\ldots, n)$.
		\item For $1 \le i,j \le n$ we denote by $E_{ij}\in M_n$ the standard matrix unit with $1$ at the position $(i,j)$ and $0$ elsewhere.
		\item The standard basis  vectors of $\C^n$ are denoted by $e_1, \ldots, e_n$.
	\end{itemize}

	\smallskip
	
	As usual, we will frequently identify vectors $x = (x_1,\ldots,x_n) \in \C^n$ as column-matrices $$x = \begin{bmatrix}
		x_1 \\ \vdots \\ x_n
	\end{bmatrix}$$ and $x^t = \begin{bmatrix} x_1 & \cdots & x_n \end{bmatrix}$ as row-matrices.\smallskip
	
	As any matrix  $A = [ a_{ij}] \in M_n$ can be understood as a map $\{1,\ldots, n\}^2 \to \C, (i,j) \mapsto a_{ij}$, we consider its support $\supp A$ as the set of all indices $(i,j) \in \{1,\ldots, n\}^2$ such that $a_{ij} \ne 0$. We also say that $A$ is supported in a set $S\subseteq \{1,\ldots, n\}^2$ if $\supp A \subseteq S$.
	
	\begin{lemma}\label{le:the J maneuver}
		Let
		$$J := \begin{bmatrix}
			0 & 0 & \cdots & 0 & 1 \\
			0 & 0 & \cdots & 1 & 0 \\
			\vdots & \vdots & \ddots & \vdots & \vdots \\
			0 & 1 & \cdots & 0 & 0 \\
			1 & 0 & \cdots & 0 & 0 \\
		\end{bmatrix} \in M_n.$$
		Then the map $\mathcal{A}_{k_1,\dots,k_r} \to \mathcal{A}_{k_r,\dots,k_1}, X \mapsto JX^t J$ is an algebra antiisomorphism.
	\end{lemma}
	\begin{proof} From $J^2=I$ and the fact that transposition is antimultiplicative we conclude that the (obviously injective and additive) map $\mathcal{A}_{k_1,\dots,k_r} \to M_n$ is an algebra antihomomorphism.  
		It remains  to show that $JX^t J \in \mathcal{A}_{k_r, \ldots, k_1}$ for all $X \in \mathcal{A}_{k_1,\ldots,k_r}$. Indeed, for each $0 \le s \le r-1$, $1 \le i \le k_1+\cdots + k_{s+1}$ and $k_1+\cdots+ k_s+1 \le j \le n$ we have
		$$JE_{ij}^tJ = E_{n+1-j,n+1-i} \in \mathcal{A}_{k_r, \ldots, k_1}$$
		since
		$k_r+\cdots+k_{s}+1 \le n+1-i \le n$ and $1 \le n+1-j \le k_r+\cdots + k_{s+1}$.
	\end{proof}
	
	This map is so ubiquitous that we will introduce a notation 
	$$
	X^{\symm} := JX^tJ
	$$ 
	for any $X \in M_n$. Note that it actually corresponds to mirroring the matrix $X$ along its secondary diagonal. Also, for a block upper-triangular algebra $\mathcal{A} \subseteq M_n$ we denote by $\mathcal{A}^\symm \subseteq M_n$ the image of $\mathcal{A}$ under the map $X \mapsto X^{\symm}$. Then $\mathcal{A}^\symm$ is the block upper-triangular algebra obtained from $\mathcal{A}$ by reversing the sizes of the diagonal blocks.
	
	\smallskip

	\begin{lemma}\label{lem:diagonalizes within A}
		Let $A \in \mathcal{A}$ be a matrix with $n$ distinct eigenvalues. Then $A$ is diagonalizable within the algebra $\mathcal{A}$. Furthermore, if $A \leftrightarrow E_{ss}$ for some $1 \le s \le n$, then the similarity matrix $T=[t_{ij} ]\in \mathcal{A}^\times$ can be chosen to satisfy $T \leftrightarrow E_{ss}$ and $t_{ss}=1$.
	\end{lemma}
	\begin{proof}
		If $\mathcal{A}=M_n$, the first statement is well-known and for $n=1$, there is nothing to do.  If $A\in  M_n$, $n\ge2$, commutes with $E_{11}$,  then 
		\begin{equation*}
			A=\begin{bmatrix}
				1 & 0 \\ 0 & A_{22}
			\end{bmatrix} =
			\begin{bmatrix}
				1 & 0 \\ 0 & S_2
			\end{bmatrix}
			\begin{bmatrix}
				1 & 0 \\ 0 & D_2
			\end{bmatrix}
			\begin{bmatrix}
				1 & 0 \\ 0 & S_2
			\end{bmatrix}^{-1},
		\end{equation*}
		where $A_{22},S_2,D\in M_{n-1}$, $D_2$ is diagonal and $S_2$ invertible. Thus the second claim has been verified. If for some $s\neq 1$, $A\leftrightarrow E_{ss}=PE_{11}P^t$, where $P$ is a suitable permutation matrix, then $P^tAP \leftrightarrow E_{11}$ and by the above argument, $P^tAP = TDT^{-1}$ for some invertible $T$ commuting with $E_{11}$ which has $1$ in the $(s,s)$-position. Now, 
		\begin{equation*}
			A= PTDT^{-1}P^t=(PTP^t)PDP^t (PTP^t)^{-1}	
		\end{equation*}
		where $PDP^t $ is a diagonal matrix with only permuted diagonal entries. It is a straightforward verification that $PTP^t$ commutes with $E_{ss}$ and has $1$ in the $(s,s)$-position.
		
		\smallskip 
		
		Let us further have $\mathcal{A}\subsetneq M_n$, $n\geq 2$.  We continue by induction on $n$. Suppose that the statements of the lemma are true for any $ 1\le m<n$.  The algebra $\mathcal{A}$ can be split so that any $A\in \mathcal{A}$  can be  written
		\begin{equation}\label{eq:x100}
			A=\begin{bmatrix}
				A_{11} & A_{12} \\
				0    & A_{22}
			\end{bmatrix},
		\end{equation}
		where $A_{11}\in M_{k_1}$, $A_{12} \in M_{k_1,n-k_1}$, and $A_{22}$ is a member of a block-upper-triangular subalgebra $\mathcal{A}_2$ of $M_{n-k_1}$. By the induction hypothesis, there exist invertible matrices $S_1\in  M_{k_1}$ and $S_2\in \mathcal{A}_2$, as well as diagonal matrices $D_1\in M_{k_1}$, $D_2\in \mathcal{A}_2$ such that $A_{11} =S_1D_1S_1^{-1}$ and $A_{22}=S_2D_2S_2^{-1}$ such that
		\begin{equation*}
			A=\begin{bmatrix}
				S_{1} & 0 \\
				0    & S_2
			\end{bmatrix}
			\begin{bmatrix}
				D_{1} & S_1^{-1}A_{12}S_2 \\
				0    & D_2
			\end{bmatrix}
			\begin{bmatrix}
				S_{1}^{-1}& 0 \\
				0    & S_2^{-1}
			\end{bmatrix}.
		\end{equation*}
		If $A$ commutes with some $E_{ss}$, then either $A_{11}\leftrightarrow E_{ii}\in M_{k_1}$ or  $A_{22}\leftrightarrow E_{jj}\in \mathcal{A}_2$ for some $j$. By the above argument in the first case or, by the induction hypothesis in the second case, either $S_1\leftrightarrow E_{ii}$  and $e_i^tS_1e_i=1$ or, $S_2\leftrightarrow E_{jj}$ with $e_j^tS_2e_j=1$.  Therefore, the block diagonal matrix $\diag(S_1, S_2) $ commutes with $ E_{ss}$ and so, there is no loss of generality in assuming that $S=I$.\smallskip
		
		Now notice that the first claim would be proved if we could write \begin{equation}\label{eq:99}
			\begin{bmatrix}
				D_{1} & A_{12}\\
				0    & D_2
			\end{bmatrix} = \begin{bmatrix}
				D_{1} & XD_2-D_1X\\
				0    & D_2
			\end{bmatrix} = 
			\begin{bmatrix}
				I_{k_1} & X\\
				0    & I_{n-k_1}
			\end{bmatrix}
			\begin{bmatrix}
				D_{1} & 0\\
				0    & D_2
			\end{bmatrix}
			\begin{bmatrix}
				I_{k_1} & X\\
				0    & I_{n-k_1}
			\end{bmatrix}^{-1}.
		\end{equation}
		for some $X \in M_{k_1,n-k_1}$. The equation $XD_2-D_1X=A_{12}$ in $X \in M_{k_1,n-k_1}$ is a Sylvester equation \cite{Horn-Johnson} and, since $\sigma(D_1)\cap \sigma(D_2)$ is empty, this equation has a unique solution $X$. To show the second claim, it remains to see that the $i$-th row ($j$-th column, resp.) of $X$ is equal to zero. We will see that with respect to the two cases, $E_{ii}X=0$ or  $XE_{jj}=0$. Note that \begin{align} \label{eq:x101}
			\begin{bmatrix}
				D_{1} & A_{12}\\
				0    & D_2
			\end{bmatrix} \begin{bmatrix}
				E_{ii} & 0\\
				0    & 0
			\end{bmatrix} - \begin{bmatrix}
				E_{ii} & 0\\
				0    & 0
			\end{bmatrix}\begin{bmatrix}
				D_{1} & A_{12}\\
				0    & D_2
			\end{bmatrix} &=\begin{bmatrix}
				0 & -E_{ii}A_{12}\\
				0    & 0 
			\end{bmatrix},   \\   
			\begin{bmatrix} \label{eq:x102}
				D_{1} & A_{12}\\
				0    & D_2
			\end{bmatrix} \begin{bmatrix}
				0 & 0\\
				0    & E_{jj}
			\end{bmatrix} - 
			\begin{bmatrix}
				0 & 0\\
				0    & E_{jj}
			\end{bmatrix}\begin{bmatrix}
				D_{1} & A_{12}\\
				0    & D_2
			\end{bmatrix} &=\begin{bmatrix}
				0 & A_{12}E_{jj}\\
				0    & 0
			\end{bmatrix}.  
		\end{align}
		If $A\leftrightarrow E_{ii}$, then from \eqref{eq:x101} we observe that
		\begin{equation*}
			0=E_{ii}A_{12} =E_{ii}(XD_2-D_1X)=(E_{ii}X)D_2-D_1(E_{ii}X).
		\end{equation*}
		Solving the Sylvester equation again, this time for $E_{ii}X$, we get $E_{ii}X=0$ as desired. In the second case, from \eqref{eq:x102} we get
		\begin{equation*}
			0=A_{12}E_{jj} =(XD_2-D_1X)E_{jj}=(XE_{jj})D_2-D_1(XE_{jj}),
		\end{equation*}
		providing $XE_{jj}=0$.
	\end{proof}

	\smallskip
	
	At the end of this preliminary section, we state a few basic properties of Jordan homomorphisms. For an algebra $\mathcal{A}$, as usual, we denote the commutator of $a,b \in \mathcal{A}$ as $[a,b] = ab-ba$. Proofs of (a), (b), and (c) of the following remark are elementary and can be found in \cite{JacobsonRickart}. 
	
	\begin{remark}\label{le:Jordan homomorphism basic properties}
		Let $\phi : \mathcal{A} \to \mathcal{B}$ be a Jordan homomorphism between complex algebras $\mathcal{A}$ and $\mathcal{B}$. We have:
		\begin{enumerate}[(a)]
			\item $\phi(aba) = \phi(a)\phi(b)\phi(a)$  for all $a,b \in \mathcal{A}$.
			\item $\phi([[a,b],c]) = [[\phi(a),\phi(b)],\phi(c)]$ for all $a,b,c \in \mathcal{A}$.
			\item $\phi([a,b]^2) = [\phi(a),\phi(b)]^2$ for all $a,b \in \mathcal{A}$.
			\item Obviously $\phi$ preserves idempotents, i.e.\ for any idempotent $a \in \mathcal{A}$, the element $\phi(a)$ is an idempotent in $\mathcal{B}$.
			
			\item If the image of $\phi$ has a trivial commutant in $\mathcal{B}$ (i.e.\ if $b \in \mathcal{B}$ commutes with every element of the image of $\phi$, then $b$ is zero or a multiple of the unity in $\mathcal{B}$, if it exists), then by (b) and (c) one easily sees that $\phi$ preserves commutativity.
			\item Suppose $\mathcal{A}$ and $\mathcal{B}$ are algebras with unities $1_\mathcal{A}$ and $1_{\mathcal{B}}$, respectively. Let $\phi : \mathcal{A} \to \mathcal{B}$ be a Jordan homomorphism such that $1_{\mathcal{B}}$ is in the image of $\phi$. Then $\phi$ is unital and preserves inverses in the sense that for every $a \in \mathcal{A}^\times$ we have $\phi(a) \in \mathcal{B}^\times$ and $\phi(a^{-1}) = \phi(a)^{-1}$ (\cite[Thm 2.5]{Fosner}, in Slovenian). Indeed, suppose that $a_0 \in \mathcal{A}$ satisfies $\phi(a_0) = 1_\mathcal{B}$. Then
			$$2\cdot 1_\mathcal{B} = \phi(a_0 + a_0) = \phi(1_{\mathcal{A}} \circ a_0) = \phi(1_{\mathcal{A}}) \circ \phi(a_0) = \phi(1_{\mathcal{A}}) \circ 1_\mathcal{B} = 2\phi(1_{\mathcal{A}})$$
			which implies $\phi(1_{\mathcal{A}}) = 1_\mathcal{B}$ so $\phi$ is unital.
			Let $a \in \mathcal{A}^\times$ be arbitrary. Then (a) gives
			$$\phi(a) = \phi(aa^{-1}a) = \phi(a)\phi(a^{-1})\phi(a).$$
			If we set $p_1 := \phi(a)\phi(a^{-1})$ and $p_2 := \phi(a^{-1})\phi(a)$, it is immediate that $p_1$ and $p_2$ are idempotents in $\mathcal{B}$. Furthermore, we have
			$$p_1+p_2 = \phi(a)\circ \phi(a^{-1}) = 2\phi(a \circ a^{-1}) = 2\phi(1_{\mathcal{A}}) = 2\cdot 1_{\mathcal{B}}$$
			and hence $2\cdot 1_{\mathcal{B}}-p_1 = p_2$ is idempotent as well. It follows that $2(p_1-1_{\mathcal{B}}) = 0$ so we conclude $p_1 = 1_{\mathcal{B}}$ and then $p_2 = 1_\mathcal{B}$ as well. Therefore, $\phi(a)$ is invertible in $\mathcal{B}$, with the inverse being equal to $\phi(a^{-1})$. In particular, all linear unital Jordan homomorphisms between unital algebras are spectrum preserving.
		\end{enumerate}
	\end{remark}

	\section{Block upper-triangular subalgebras and their Jordan embeddings}\label{sec:block upper-triangular subalgebras and their Jordan embeddings}
	
	We start this section by proving Theorem \ref{thm:Jordan monomorphisms on A}.

	\begin{proof}[Proof of Theorem \ref{thm:Jordan monomorphisms on A}]
		We prove the theorem by induction on $n$. For $n=1$ the statement is clear. So suppose it holds for all $k$,  $1\le k < n$, and let $\mathcal{A}\subseteq M_n$ be a block upper-triangular algebra. As $\phi$ is a Jordan homomorphism, it preserves idempotents. For every distinct $i,j\in\{1,2,\ldots ,n\}$ we have
		$0=\phi(E_{ii}\circ E_{jj})=\phi(E_{ii})\circ \phi(E_{jj})$ and so, $\phi(E_{ii})$, $i=1,2,\dots, n$, is a family of $n$ pairwise orthogonal idempotents. This follows from a general fact: if $P$, $Q \in M_n$ are idempotents such that $P\circ Q = 0$, then $P \perp Q$. Indeed, we have $PQ+QP = 0$, so by multiplying this by $P$ from both sides, one gets $PQP = 0$. As $QP = -PQ$, this implies $PQ =-PQP = 0$. Then also $PQ = -QP = 0$. It follows that there exists an invertible matrix $S$ such that $\phi(E_{ii})=SE_{ii}S^{-1}$ for all $1 \le i \le n$. With no loss of generality we assume that $\phi(E_{11})=E_{11}$ and in turn,  $SE_{11}=E_{11}S$. For every matrix $A\in \mathrm{span}\{E_{ij} : 2 \le i,j\le n\}\cap \mathcal{A}$ we have that $A \circ E_{11}=0$ and hence $\phi(A)\circ E_{11}=0$, which implies $\phi(A) \perp E_{11}$. So, we can define a Jordan embedding $\psi: \mathcal{A}_{n-1} \to M_{n-1}$, where $\mathcal{A}_{n-1} \subseteq  M_{n-1}$ is a block upper-triangular subalgebra, by 
		$$\phi \left(\begin{bmatrix}
			0 & 0\\
			0 & A
		\end{bmatrix}\right) =\begin{bmatrix}
			0 & 0\\
			0 & \psi(A)
		\end{bmatrix}.$$
		Now we apply the induction hypothesis to get that there exists an invertible $T\in M_{n-1}$ such that either $\psi(A)=TAT^{-1}$ for all $A\in \mathcal{A}_{n-1}$ or, $\psi(A)=TA^{t}T^{-1}$ for all $A\in \mathcal{A}_{n-1}$. By passing to the map $\diag(1,T)^{-1}\phi(\,\cdot\,)\diag(1,T)$, we can further take $T$ to be the identity matrix in $M_{n-1}$ and therefore, $\phi$ fixes all diagonal matrices. In particular, $\phi(E_{ii})=E_{ii}$, for all $1\le i \le n$. By linearity, it suffices to check that there is an invertible (diagonal in fact by our last assumption) matrix $T\in M_n$ such that $\phi(E_{ij}) = TE_{ij}T^{-1}$ for all $1 \le i,j \le n$ or, $\phi(E_{ij}) = TE_{ji}T^{-1}$ for all $1 \le i,j \le n$. \smallskip
		
		Suppose that $E_{ij}\in\mathcal{A}$, $i\neq j$. We claim that $\phi(E_{ij})=\alpha_{ij}E_{ij}+\beta_{ij}E_{ji}$,   where exactly one of $\alpha_{ij}$ and $\beta_{ij}$ is non-zero. When $n>2$, we first use that for every $k\notin \{i,j\}$, $0=\phi(E_{ij}\circ E_{kk})=\phi(E_{ij})\circ E_{kk}$. So, for $n\geq 2$ we have   $\phi(E_{ij})\in \mathrm{span}\{E_{ii},E_{ij},E_{ji},E_{jj}\}$. From $0=\phi(E_{ij}\circ E_{ij})=2\phi(E_{ij})^2$ we see that $\phi(E_{ij})$ is a rank-one nilpotent as $\phi$ is injective.
		Now, the claim follows from
		$$\phi(E_{ij})=\phi(E_{ij}\circ E_{ii})=\phi(E_{ij})\circ E_{ii},$$
		since this means that $\phi(E_{ij})$ is supported in the $i$-th row and column. Therefore, $\phi(E_{ij})_{jj} = 0$ and hence the other eigenvalue $\phi(E_{ij})_{ii}$ also has to be zero.\smallskip
		
		So, further suppose $n\ge 3$. If $\psi(A) = A$ for every $A\in \mathcal{A}_{n-1}$ then we claim that $\phi(E_{1j})=\alpha_{1j}E_{1j}$ for all $1 < j \le n$. If not, first suppose that $\phi(E_{1k})=\beta_{1k}E_{k1}$ for some $1<k<n$.  Then $\phi(E_{1k}\circ E_{kn})=\phi(E_{1n})$ but $\phi(E_{1k})\circ \phi(E_{kn})=\beta_{1k}E_{k1}\circ E_{kn}=0$, a contradiction. Secondly, the case $\phi(E_{1n}) = \beta_{1n}E_{n1}$ can be eliminated by comparing $\phi(E_{1n}\circ E_{1,n-1})=0$ and $\phi(E_{1n})\circ \phi(E_{1,n-1})=\beta_{1n}E_{n1} \circ E_{1,n-1}=\beta_{1n}E_{n,n-1}$. By a diagonal similarity implemented by $\diag(1,\alpha_{12},\ldots,\alpha_{1n}) \in \mathcal{D}_n^\times$, we can achieve that $\phi(E_{1j})=E_{1j}$ for every $j=2,\dots, n$.
		
		\smallskip
		If $\mathcal{A}=\mathcal{T}_n$, we are done, else, for every $E_{j1}\in \mathcal{A}$ with $2 \le j \le n$,  we have that $\phi(E_{1j}\circ E_{j1})=E_{11}+E_{jj}$, so $E_{1j}\circ (\alpha_{j1}E_{j1}+\beta_{j1}E_{1j})=\alpha_{j1}(E_{11}+E_{jj})$	 gives $\alpha_{j1}=1$ and hence  $\beta_{j1}=0$.\smallskip
		
		The antimultiplicative case of $\psi$ can be considered by replacing the map $\phi$ with the map $\phi(\cdot)^t$. Finally, the linearity of $\phi$ closes the proof.
	\end{proof}

	\begin{corollary}\label{cor:Jordan monomorphism unified}
		Let $\mathcal{A}$ and $\mathcal{B}$ be block-upper-triangular subalgebras of $M_n$. Suppose that $\phi : \mathcal{A} \to \mathcal{B}$ is a Jordan embedding. Then one of the following is true:
		\begin{enumerate}[(a)]
			\item $\mathcal{A} \subseteq \mathcal{B}$ and there exists $T \in \mathcal{B}^{\times}$ such that $\phi(X) = TXT^{-1}$ for all $X \in \mathcal{A}$,
			\item $\mathcal{A}^{\symm} \subseteq \mathcal{B}$ and there exists $T \in \mathcal{B}^{\times}$ such that $\phi(X) = TX^{\symm}T^{-1}$ for all $X \in \mathcal{A}$.
		\end{enumerate}
	\end{corollary}
	\begin{proof}
		We will show that $E_{ij}\in\mathcal{A}$ implies $E_{ij}\in\mathcal{B}$ for all $E_{ij}\in\mathcal{A}$ or, $E_{ij}\in\mathcal{A}^\odot$ implies $E_{ij}\in\mathcal{B}$ for all $E_{ij}\in\mathcal{A}^\odot$.  We again use mathematical induction on $n$. There is nothing to do when $n=1$, so for the induction step, assume $n\ge 2$ and that the claim already holds for $n-1$. Let  $\mathcal{A}=\mathcal{A}_{k_1,\dots,k_p}$ and $\mathcal{B}=\mathcal{A}_{l_1,\dots,l_q}$.
		
		\smallskip 
		Assume that $\phi$ is of the form $\phi(A)=TAT^{-1}$ for some fixed invertible matrix $T\in M_n$.

		By Lemma \ref{lem:diagonalizes within A} there exists an $S\in \mathcal{B}^\times$ and a diagonal matrix $\Delta^\prime$, having  the diagonal entries of $\Delta$  possibly permuted, such that $T\Delta T^{-1}=\phi(\Delta)=S\Delta^\prime S^{-1}$. By the commutativity preserving, $ S^{-1}\phi(E_{ii})S\leftrightarrow \Delta^\prime$ for each $ i=1,2,\dots,n$, and so, $\{ S^{-1}\phi(E_{ii})S: i=1,2,\dots,n\}$ is a family of pairwise commuting diagonal rank-one idempotent matrices. It follows that $\phi(E_{ii})=SE_{\pi(i)\pi(i)}S^{-1}$ for some permutation $\pi$ of the set $\{1,2,\dots,n\}$. There is a unique permutation matrix $P\in M_n$ such that $PE_{ii}P^t=E_{\pi(i)\pi(i)}=Pe_i(Pe_i)^t$ for all $1 \le i \le n$. Then 
		\begin{equation*}
			T\Delta T^{-1}=\phi(\Delta)= \phi\left(\sum_{i=1}^n iE_{ii}\right)= S \left(\sum_{i=1}^n iPE_{ii}P^t\right) S^{-1} =SP\Delta (SP)^{-1}
		\end{equation*}
		and consequently, $(SP)^{-1}T$ commutes with $\Delta$, whence we infer that $T=SPD$ for some invertible diagonal matrix $D=\diag(d_1,\dots,d_n)$.  This implies that 
		\begin{align*}
			\phi(E_{ij}) &= TE_{ij}T^{-1}=SPDE_{ij}D^{-1}P^tS^{-1}=S\left(\frac{d_i}{d_j}Pe_i(Pe_j)^t\right) S^{-1} \\
			&=\frac{d_i}{d_j} S E_{\pi(i)\pi(j)}S^{-1}
		\end{align*}	
		for every $1\le i,j \le n$ such that $E_{ij} \in \mathcal{A}$.

		As $S\in\mathcal{B}^\times$, replacing $\phi$ by $X\mapsto S^{-1}\phi(X)S$  we can further suppose that $S=I$ and $\phi(E_{ij})=d_i/d_j E_{\pi(i)\pi(j)} $ for all $i,j$.  For each fixed $i$,  $1\le i\le k_1$, $\spn\{E_{ij} : j=1,\dots, n\}$ is an $n$-dimensional vector space of rank-one matrices being supported in only one row, which is mapped into $\spn\{E_{\pi(i)s}: s=1,\dots, n\}$. As $\phi$ is linear and injective, we have $\pi(i)\in \{1,2,\dots,l_1\}$ where $l_1\ge k_1$. Let $Q \in M_n$ be the permutation matrix which swaps $1$ and $\pi(i)$. Since $1 \le \pi(i) \le l_1$, we have $Q \in \mathcal{B}$. Therefore, by conjugating the map $\phi$ by $Q$, we can achieve that all matrices supported only in the first row are mapped by $\phi$ to the matrices of the same kind, i.e.\ $\pi(1)=1$ with no loss of generality. For all $E_{i1}\in \mathcal{A}$  we have that $\phi(E_{i1})=d_i/d_1E_{\pi(i)1}\in \mathcal{B}$ since $\pi(i)\le l_1$.
		
		\smallskip
		Now we can define a Jordan embedding $\psi: \mathcal{A}^\prime \to \mathcal{B}^\prime$, where $\mathcal{A}^\prime ,\, \mathcal{B}^\prime  \subseteq M_{n-1}$ are block upper-triangular algebras, obtained by omitting the first row and the first column in $\mathcal{A}$ and $\mathcal{B}$, by 
		$$\phi \left(\begin{bmatrix}
			0 & 0\\
			0 & A
		\end{bmatrix}\right) =\begin{bmatrix}
			0 & 0\\
			0 & \psi(A)
		\end{bmatrix} \in \mathcal{B},\ \ \ A\in\mathcal{A}^\prime,$$
		which, by the induction hypothesis, gives that $\mathcal{A}^\prime\subseteq \mathcal{B}^\prime$ so, for every  $E_{ij} \in \mathcal{A}$ we have that $E_{ij} \in \mathcal{B}$  for all $i,j$ such that $i,j\neq 1$. By linearity and the fact that $\phi$ also maps the matrices supported in the first row and those supported in the first column to matrices of the same form, we deduce $\mathcal{A}\subseteq \mathcal{B}$. Additionally, applying $\psi(A) =T^\prime A {T^\prime}^{-1}$ gives that  
		\begin{equation*}
			\begin{bmatrix}
				1 & 0\\
				0 & {T^\prime}^{-1}
			\end{bmatrix}
			T\begin{bmatrix}
				0 & 0\\
				0 & A
			\end{bmatrix} 
			=
			\begin{bmatrix}
				0 & 0\\
				0 & A
			\end{bmatrix}
			\begin{bmatrix}
				1 & 0\\
				0 &{T^\prime}^{-1}
			\end{bmatrix}T,\ \ 	\text{for every } A\in \mathcal{A}^\prime, 
		\end{equation*}
		where $T^\prime\in {\mathcal{B}^\prime}^\times$. The commutativity relation above provides
		$$
		T=\begin{bmatrix}
			1 & 0\\
			0 &T^\prime
		\end{bmatrix} D
		$$
		for some invertible diagonal matrix $D\in \mathcal{B}$. We have thus obtained that $T\in \mathcal{B}^\times$.
		
		\smallskip 
		For the remaining case, when $\phi$ is antimultiplicative, we may instead consider the multiplicative linear map $A\mapsto \phi(A^\odot) $, $A\in \mathcal{A}^\odot$, which ends the proof by applying the above consideration. 
	\end{proof}

	The next simple example shows that the image of a Jordan embedding $\phi: \mathcal{A}\to M_n$, where $\mathcal{A}$ is a block-upper-triangular subalgebra of $M_n$, needs not to be a block-upper-triangular subalgebra of $M_n$.
	\begin{example}
		Consider the algebra monomorphism $\phi : \mathcal{T}_3 \to M_3$ given by $\phi(X) = JXJ^{-1}$ where $J \in M_3^\times$ is from Lemma \ref{le:the J maneuver}.
		The image of $\phi$ is precisely the algebra of $3 \times 3$ lower-triangular matrices, hence not equal to $\mathcal{T}_3$ nor $\mathcal{T}_3^{\odot} = \mathcal{T}_3$.
	\end{example}

		\begin{corollary}\label{cor:when are block upper-triangular algebras isomorphic}
			Let $\mathcal{A}=\mathcal{A}_{k_1,\ldots,k_p}$ and $\mathcal{B}=\mathcal{A}_{l_1,\ldots,l_q}$ be block upper-triangular subalgebras of $M_n$.
			\begin{enumerate}
				\item[(a)] $\mathcal{A}$ and $\mathcal{B}$ are algebra-isomorphic if and only if  $(k_1,\ldots,k_p) = (l_1,\ldots,l_q)$.
				\item[(b)] $\mathcal{A}$ and $\mathcal{B}$ are algebra-antiisomorphic if and only if $(k_1,\ldots,k_p) = (l_q,\ldots,l_1)$.
				\item[(c)] $\mathcal{A}$ and $\mathcal{B}$ are Jordan-isomorphic if and only if $(k_1,\ldots,k_p) = (l_1,\ldots,l_q)$ or  $(k_1,\ldots,k_p)=(l_q,\ldots,l_1)$.
			\end{enumerate}
		\end{corollary}
		\begin{proof}
			We write the proofs only for the nontrivial implications.
			\smallskip
			
			(a) 
			Let $\phi :\mathcal{A} \to \mathcal{B}$ be an algebra  isomorphism. Then  $\phi$ is a multiplicative Jordan isomorphism and Corollary \ref{cor:Jordan monomorphism unified} implies $\mathcal{A} \subseteq \mathcal{B}$. As the same argument applies to the (multiplicative) map $\phi^{-1}$, we obtain $\mathcal{A} = \mathcal{B}$, i.e.\ $(k_1,\ldots,k_p) = (l_1,\ldots,l_q)$.\smallskip

			(b) For an algebraic antiisomorphism $\phi:\mathcal{A} \to \mathcal{B}$ which is antimultiplicative, Corollary \ref{cor:Jordan monomorphism unified} implies $\mathcal{A}^{\odot} \subseteq \mathcal{B}$. Applying the same argument to $\phi^{-1}$ yields $\mathcal{B}^{\odot} \subseteq \mathcal{A}$. Since the map $\odot$ is involutory, we have $\mathcal{B} \subseteq \mathcal{A}^{\odot}$ and therefore $\mathcal{A}^{\odot} = \mathcal{B}$ (which is equivalent to $\mathcal{A}= \mathcal{B}^{\odot} $). Consequently,  $(k_1,\ldots,k_p) = (l_q,\ldots,l_1)$.\smallskip
			
			(c) By Corollary \ref{cor:Jordan monomorphism unified} a Jordan isomorphism $\phi:\mathcal{A} \to \mathcal{B}$ is multiplicative or antimultiplicative, so the arguments in (a) and (b) can be applied.
			
		\end{proof}
		
		\section{Proof of the main result}\label{sec:Proof of the main result}
		
		We now proceed with the proof of our main result, namely Theorem \ref{thm:main result}.  We begin with the following auxiliary fact:
		
		\begin{lemma}\label{lem:char}
			Let $\mathcal{A}$ be a block upper-triangular subalgebra of $M_n$ and  $\phi:\mathcal{A}\to M_n$ a continuous map satisfying  $\sigma(A)\subseteq \sigma(\phi(A))$ for every matrix $A$ in some open (relative to $\mathcal{A}$) set $\,\mathcal{U} \subseteq \mathcal{A}$. Then $p_A=p_{\phi(A)}$ for all $A\in \mathcal{U}$, thus $\phi$ is spectrum preserving on $\,\mathcal{U}$ and in turn, the algebraic multiplicities of all eigenvalues are also preserved by $\phi$.  	
		\end{lemma}	
		
		\begin{proof} The assertion is clearly true on every  set $\mathcal{E}\subseteq \mathcal{U}$ of all matrices  with $n$ distinct eigenvalues. 
			By applying the Schur triangularization on each diagonal block, every matrix $A\in\mathcal{U}$ can be represented as $A=S(\Lambda+N)S^{-1}$ for some diagonal matrix $\Lambda$, an invertible (block-diagonal) matrix $S\in\mathcal{A}$ and a strictly upper-triangular matrix $N$. Define a family of matrices   $A_v=S(\Lambda+E(v)+N)S^{-1}$  where $E(v)=\mathrm{diag}(v_1,\dots,v_n)$, and the diagonal entries of $\Lambda+E(v) =\mathrm{diag}(\lambda_1+v_1,\dots,\lambda_n +v_n)$ are all distinct whenever $\norm{v}=\norm{ (v_1,\dots,v_n)}$ is sufficiently small.
			Clearly, $A=\lim_{v\to 0} A_v$ and by the continuity of $\phi$ and of the determinant
			\begin{align*}
				p_A(x)&=\mathrm{det}(A-xI) =\lim_{v\to 0}\mathrm{det}(A_v-xI) \\
				&=\lim_{v\to 0}\mathrm{det}(\phi(A_v)-xI)=\mathrm{det}(\phi(A)-xI)=p_{\phi(A)}(x)
			\end{align*}
			confirming the assertion of the lemma.
		\end{proof}

		The following lemma might be well-known, but we include it for completeness.
		\begin{lemma}\label{lem:idR}
			For every  idempotent rank-one matrix $R\in \mathcal{T}_n$ there exists a matrix $T\in \mathcal{T}_n^\times$ such that $R=TE_{ii}T^{-1}$ for some $i\in \{1,2,\dots,n\}$.
		\end{lemma}
		\begin{proof}
			
			We prove the claim by induction on $n$. There is nothing to do when $n=1$. 
			Suppose further that the claim holds for $n-1\ge 1$ and that the last row of $R\in \mathcal{T}_{n}$ vanishes. Representing
			\begin{equation*}
				R=\begin{bmatrix}
					R_{1} & R_2 \\ 0 & 0
				\end{bmatrix}
			\end{equation*}
			where $R_1\in \mathcal{T}_{n-1}$ and applying the induction hypothesis gives that $R_1=T_1 E_{ii} T_1^{-1}$ for some $T_1\in \mathcal{T}_{n-1}^\times$ and $1\le i \le n-1$. It follows that 
			\begin{equation*}
				\diag(T_1,1)^{-1}R\diag(T_1,1)=	E_{ii} + \alpha E_{in}
			\end{equation*}
			for some $1\le i <n$ and a complex number $\alpha$. By setting $S=I+\alpha E_{in}\in \mathcal{T}_n^\times$ we easily get that $E_{ii} + \alpha E_{in}= SE_{ii}S^{-1}$. \smallskip
			
			If the first row of $R$ is equal to zero, then we apply the above argument on $R^\odot$.
		\end{proof}
		
		\smallskip
		
		We will need the following lemma concerning particular upper-triangular families of pairwise orthogonal rank-one idempotent matrices.
		
		\begin{lemma}\label{lem:techR}
			Let  $S(y):=I+e_1y^t\in \mathcal{T}_n$, where $y=\sum_{i=2}^{n}y_ie_i \in \C^n$. Then $S(y)^{-1}=S(-y)$ and
			\begin{equation} \label{al:SES}
				S(y)^{-1}E_{ii}S(y) =   
				\begin{cases}
					E_{11} +e_1y^t,   &\quad \text{if } i=1,	\\ 
					E_{ii}-y_iE_{1i}, &\quad \text{if }  1<i\le n.	\\ 
				\end{cases}
			\end{equation}
		\end{lemma}
		\begin{proof}
			It is easy to check that 	$S(y)^{-1}=S(-y)$. The straightforward calculation 
			\begin{align*}
				S(y)^{-1}E_{11}S(y)&=(S(-y)e_1)(e_1^t S(y)) =(e_1 - (y^te_1)e_1)(e_1+y)^t \\
				S(y)^{-1}E_{ii}S(y)&=(S(-y)e_i)(e_i^t S(y)) =(e_i - (y^te_i)e_1)e_i^t,\ \ i\neq 1,  		
			\end{align*}
			provides \eqref{al:SES}. 
		\end{proof}

		\begin{proof}[Proof of Theorem \ref{thm:main result}] Throughout the proof, let $\mathcal{A}:=\mathcal{A}_{k_1,k_2,\dots,k_r}$, $r\geq 2$, be fixed.
			Let $\phi : \mathcal{A}\to M_n$ be a continuous injective map which preserves commutativity and spectrum. By Lemma \ref{lem:char}, we first conclude that $\phi$ also preserves the algebraic multiplicities of the eigenvalues.
			
			\smallskip
			
			For an easier comprehension, the proof will now be divided into several steps.

			\smallskip
			\begin{step}\label{st:0001} \textit{There exists an invertible matrix $S\in \mathcal{A}$ such that
					$\phi(D) = SDS^{-1}$,  for all diagonal matrices $D \in \mathcal{A}$.  }
			\end{step}
			\begin{proof}
				Since the matrix $\phi(\Delta) \in M_n$ is diagonalizable with eigenvalues $1,\ldots, n$, there exists an $S \in M_n^{\times}$ such that $\phi(\Delta) = S\Delta S^{-1}$. For any $D \in \mathcal{D}_n$ we have  $\phi(D) \leftrightarrow \phi(\Delta) = S\Delta S^{-1}$, since $D \leftrightarrow \Delta$. So, $\phi(D)=S\widetilde{D} S^{-1}$ for some diagonal $\widetilde{D}\in M_n$. By the continuity of $\phi$, the argument from \cite[Lemma 2.1]{Semrl} gives that $\widetilde{D}=D$. For completeness, we include it here. Assume first that the diagonal entries of $D$ are all distinct, and denote them by $\lambda_1, \ldots, \lambda_n$. Choose continuous paths $f_k : [0,1] \to \C, 1 \le k \le n$  from $k$ to $\lambda_k$ such that for all $t \in [0,1]$ the values $f_1(t), \ldots,f_n(t)$ are all distinct. To be explicit, for each $1 \le k \le n$ and any path $$\alpha_k : [0,1]\to (\C\setminus\{1,\ldots,n, \lambda_1,\ldots,\lambda_n\}) \cup \{k,\lambda_k\}$$ from $k$ to $\lambda_k$ (which exists by path-connectedness), we can define
				$$f_k(t) := \begin{cases} k, \quad &\text{ if }t \in \left[0,\frac{k-1}n\right], \\
					\alpha_k\left(n\left(t-\frac{k-1}n\right)\right), \quad &\text{ if }t \in \left[\frac{k-1}n,\frac{k}n\right], \\
					\lambda_k, \quad &\text{ if } t \in \left[\frac{k}n, 1\right]. \\
				\end{cases}$$
				
				Denote $$d:=  \min_{t \in [0,1]}\Big\{\abs{f_i(t)-f_j(t)} : 1 \le i,j \le n\Big\} > 0.$$
				
				Notice that the set
				\begin{align*}
					\mathcal{S} &= \{t \in [0,1] : \phi(\diag(f_1(t),\ldots,f_n(t))) \ne S\diag(f_1(t),\ldots,f_n(t))S^{-1}\} \\
					&= \{t \in [0,1] : \norm{S^{-1}\phi(\diag(f_1(t),\ldots,f_n(t)))S - \diag(f_1(t),\ldots,f_n(t))}_\infty \ge d\}
				\end{align*}
				is both open and closed in $[0,1]$. Since $0 \notin \mathcal{S}$, by the connectedness of $[0,1]$ it follows that $\mathcal{S} = \emptyset$. In particular, for $t = 1$ we get
				\begin{align*}
					\phi(\diag(\lambda_1,\ldots,\lambda_{n}))&=\phi(\diag(f_1(1), \ldots, f_n(1))) = S\diag(f_1(1), \ldots, f_n(1))S^{-1} \\
					&=S\diag(\lambda_1,\ldots,\lambda_n)S^{-1}.
				\end{align*}
				As the diagonal matrices with distinct eigenvalues are dense in $\mathcal{D}_n$ and $\phi$ is continuous, the claim follows for all $D \in \mathcal{D}_n$.
				
			\end{proof}
			
			\medskip
			In view of Step \ref{st:0001}, by passing to the map $S^{-1}\phi(\cdot) S$ we can assume further that $\phi$ is the identity on all diagonal matrices.
			
			
			
			
			\begin{step}\label{st:hom0}
				\phantom{a}
				\begin{enumerate}[(a)]
					\item \textit{$\phi$ is a homogeneous map,}
					\item \textit{If $A$ and $B$ are diagonalizable matrices in $\mathcal{A}$, with the similarity matrix in $\mathcal{A}^\times$, satisfying $AB=BA=0$, then $\phi(A)\phi(B)=\phi(B)\phi(A)=0$.}
				\end{enumerate}
			\end{step}
			
			\begin{proof}
				\begin{enumerate}[(a)]
					\item For any $S\in \mathcal{A}^\times$ there exists a matrix $T_S\in M_n^\times$ such that $\phi(S\Delta S^{-1})=T_S\Delta T_S^{-1}$. Applying Step 1 for the map $X\mapsto T_S^{-1}\phi(SX S^{-1})T_S$ provides that $\phi(SDS^{-1})=T_SDT_S^{-1}$ for every diagonal matrix $D$. Replacing $D$ by $\alpha D$, $\alpha \in \mathbb{C}$ gives that
					$\phi$ is homogeneous on the set of all diagonalizable matrices in $\mathcal{A}$. By the density of such matrices (a consequence of Lemma \ref{lem:diagonalizes within A}) we conclude that $\phi$ is a homogeneous map on $\mathcal{A}$.
					\item Assume that $A=S_1D_1S_1^{-1}$ for some $S_1\in \mathcal{A}^\times$ and $D_1=\diag(d_1,d_2,\dots,d_n)$. If $A=0$ or $A$ is invertible, we are done. So, we can assume that the index set $\mathcal{J}=\{1 \le j \le n : \ d_j\neq  0\}$ is a nonempty proper subset of $\{1,2,\dots,n\}$. We observe that $S_1^{-1}BS_1\in \left( \cap_{j\in \mathcal{J}} E_{jj}^\perp \right) \cap \mathcal{A}$ and, since it is diagonalizable, there exists a $S_2\in \mathcal{A}^\times$ commuting with all $E_{jj}$, $j\in \mathcal{J}$, such that $S_2^{-1}S_1^{-1}BS_1S_2=:D_2$ is diagonal and $S_2^{-1}S_1^{-1}AS_1S_2=D_1$. Now, for $S:=S_1S_2\in\mathcal{A}^\times$ we have $A=SD_1S^{-1}$ and $B=SD_2S^{-1}$. To complete the proof of (b), there exists an invertible matrix $T \in M_n$ such that $\phi(S\Delta S^{-1})=T\Delta T^{-1}$. It follows that $\phi(SD_jS^{-1})=TD_j T^{-1}$, $j=1,2$. Hence $\phi(A)\phi(B)=\phi(B)\phi(A)=0$.
				\end{enumerate}
			\end{proof}
			
			\begin{step}\label{st:rk1}
				\textit{$\phi$ maps every  rank-one  matrix in $\mathcal{A}$ to a rank-one matrix in $M_n$ with the same spectrum.}
			\end{step}
			\begin{proof}		
				Let first $R\in \mathcal{A}$ be a rank-one idempotent matrix. In particular, only one diagonal block of $R$ is nonzero.  By Schur triangularization theorem the nonzero diagonal block of $R$ is unitarily similar to an upper-triangular matrix. It follows that for some block-diagonal unitary matrix $U\in \mathcal{A}$, the matrix $U^\ast RU$ is upper-triangular. Now, applying Lemma \ref{lem:idR} for some $S\in \mathcal{T}_n^\times$ and $i\in \{1,2,\dots,n\}$ we have  $R=USE_{ii}S^{-1}U^\ast$. As $\phi(US \Delta S^{-1}U^\ast)=T\Delta T^{-1}$ for some $T\in \mathcal{A}^\times$, by Step  \ref{st:0001} $$\phi(R)=\phi(USE_{ii}S^{-1}U^\ast)=TE_{ii} T^{-1}$$ is a rank-one idempotent.
				
				\smallskip
				
				By Step \ref{st:hom0} the map $\phi$ is homogeneous, so  $\phi(\alpha R)$ is of rank one for every non zero $\alpha\in \C$. Finally, every rank-one nilpotent matrix can be obtained as limit of a sequence of non-nilpotent rank-one matrices. As the rank is lower-semicontinuous, it can only decrease and, by the injectivity of $\phi$, it cannot reach zero.
			\end{proof}
			
			\begin{step} \label{st:0003}
				\textit{Fix $k\in \{1,2,\dots,n\}$. Let  $S\in \mathcal{A}$ and $T\in M_n$ be  invertible matrices such that  $\phi(SE_{ii}S^{-1})=TE_{ii}T^{-1}$ for all $i\neq k$. Then $\phi(SE_{kk}S^{-1})=TE_{kk}T^{-1}$. }
			\end{step}	
			\begin{proof}
				We know that $\phi(S\Delta S^{-1})=P\Delta P^{-1}$ for some invertible $P\in M_n$ which implies that $\phi(SDS^{-1})=PDP^{-1}$ for every diagonal matrix $D$. It follows that $TE_{ii}T^{-1}=PE_{ii}P^{-1}$ whenever $i\neq k$ and so $$\phi(SE_{kk}S^{-1})=PE_{kk}P^{-1} =I- \sum_{i\neq k} PE_{ii}P^{-1}= I- \sum_{i\neq k}TE_{ii}T^{-1}=TE_{kk}T^{-1}.$$
			\end{proof}
			%

			\begin{step}[\textbf{Base of induction, $n = 3$}]
				Now we prove Theorem \ref{thm:main result} completely for $n=3$.
			\end{step}
			It suffices to prove the theorem for $\mathcal{A} =  \mathcal{A}_{1,1,1}$ or $\mathcal{A}=\mathcal{A}_{1,2}$. Indeed, the $M_3$ case is covered by Theorem \ref{thm:Semrl}, while if $\phi : \mathcal{A}_{2,1} \to M_3$ is a continuous injective commutativity and spectrum preserving map, then the map $$X \mapsto \phi(X^{\symm}) : \mathcal{A}_{1,2} \to M_3$$
			inherits similar
			properties from $\phi$. Thus the theorem follows from the $\mathcal{A}_{1,2}$ case.
			
			\begin{stepa}[\textbf{Matrix units}]\label{step:behaviour on matrix units 3x3}
				We show that we can, without loss of generality, assume that there exist constants $c_{ij} \in \C^{\times}$ such that $$\phi(E_{ij}) = c_{ij} E_{ij}, \qquad \text{ for all matrix units }E_{ij} \in \mathcal{A}.$$
			\end{stepa}
			We have $E_{12} \leftrightarrow E_{33}$ so by our assumption that $\phi$ fixes all diagonal matrices, we have $\phi(E_{12}) \leftrightarrow \phi(E_{33}) = E_{33}$. By the fact that $\phi$ is injective and preserves spectrum, the matrix $\phi(E_{12})$ is a nonzero nilpotent so we conclude
			$$\phi(E_{12}) = \begin{bmatrix} ac & a^2 & 0 \\ -c^2 & -ac & 0 \\ 0 & 0 & 0\end{bmatrix}$$
			for some $a,c \in \C$ not both equal to $0$. Analogously we get
			$$\phi(E_{13}) = \begin{bmatrix} bd & 0 & b^2 \\ 0 & 0 & 0 \\ -d^2 & 0 & -bd\end{bmatrix}$$
			for some $b,d \in \C$ not both equal to $0$. Since $E_{12} \leftrightarrow E_{13}$ we have
			$$\begin{bmatrix}
				a b c d & 0 & a b^2 c \\
				-b c^2 d & 0 & -b^2 c^2 \\
				0 & 0 & 0 \\
			\end{bmatrix} = \phi(E_{12})\phi(E_{13}) = \phi(E_{13})\phi(E_{12}) = \begin{bmatrix}
				a b c d & a^2 b d & 0 \\
				0 & 0 & 0 \\
				-a c d^2 & -a^2 d^2 & 0 \\
			\end{bmatrix}.$$
			If $a \ne 0$, then $-a^2d^2 = 0$ implies $d =0$ so $b \ne 0$. Now $ab^2c = 0$ implies $c = 0$ and therefore
			$$\phi(E_{12}) = a^2E_{12}, \quad \phi(E_{13}) = b^2 E_{13}.$$
			If $c \ne 0$, then $-b^2c^2 = 0$ implies $b = 0$ so $d \ne 0$. Now $a^2bd = 0$ implies $a = 0$ and therefore
			$$\phi(E_{12}) = -c^2E_{21}, \quad \phi(E_{13}) = -d^2 E_{31}.$$
			
			Without loss of generality, we can suppose that we are in the first case, as otherwise we can pass to the map $\phi(\cdot)^t$.
			
			\smallskip
			
			Following analogous reasoning, we also obtain
			$$\phi(E_{23}) = \begin{bmatrix} 0 & 0 & 0 \\ 0 & ef & e^2 \\ 0 & -f^2 & -ef\end{bmatrix}$$
			for some $e,f \in \C$ not both equal to $0$.
			By using $E_{23} \leftrightarrow E_{13}$ it follows that $f = 0$ and therefore $\phi(E_{23}) = e^2 E_{23}$.\\
			
			This settles the matter for $\mathcal{A} = \mathcal{T}_3$. In the case of $\mathcal{A} = \mathcal{A}_{1,2}$, similarly as above using $E_{32} \leftrightarrow E_{11}, E_{12}$ we obtain $\phi(E_{32}) = c_{32}E_{32}$ for some $c_{32} \in \C^{\times}$.
			
			\smallskip
			
			As $\phi(E_{1j})=c_{1j}E_{1j}$, notice that
			$$\diag(1, c_{12},c_{13})\phi(E_{1j})\diag(1, c_{12},c_{13})^{-1}= E_{1j}, \quad 1 \le j \le 3.$$
			Therefore, the map 
			\begin{equation}\label{eq:newphi}
				\mathcal{A}\to M_3, \qquad \diag(1, c_{12},c_{13})\phi(\cdot)\diag(1, c_{12}, c_{13})^{-1}
			\end{equation}
			inherits all  the relevant properties of $\phi$ (continuity, injectivity, spectrum and commutativity preserving, as well acting as the identity on $\mathcal{D}_3$) and satisfies $E_{ij} \mapsto c'_{ij}E_{ij}$ for all $E_{ij}\in \mathcal{A}$, where $c_{ij}'=\frac{c_{1i}c_{ij}}{c_{1j}}$. Obviously, $c'_{ij}=1$ if $i=j$ or $i=1$. By replacing $\phi$ by the map defined by \eqref{eq:newphi}, without loss of generality we can further assume that $c_{1j}=1$, i.e. $\phi(E_{1j})=E_{1j}$ for all $1 \leq j \leq 3$.
			
			\begin{stepa}[\textbf{Rank-one matrices}]\label{step:identity on rank-one matrices 3x3}
				We show that $\phi(A) = A$ for all non-nilpotent rank-one matrices $A \in \mathcal{A}$.
			\end{stepa}

			When we refer to orthogonality in $\mathbb{C}^n$ we mean orthogonality with respect to the standard inner product, i.e.\ $x\perp y$ if and only if $y^\ast x=0\  (=x^\ast y)$.
			\smallskip

			Recalling the assumption $\mathcal{A}=\mathcal{A}_{1,2}$ or $\mathcal{A}=\mathcal{T}_3$, every rank-one matrix in $\mathcal{A}$ is either of the form 
			
			\begin{equation*}
				(a)\, \begin{bmatrix}
					\ast  & \ast & \ast\\
					0  & 0 & 0 \\
					0 & 0 & 0 
				\end{bmatrix} \qquad  \text{or} \qquad (b) \, \begin{bmatrix}
					0  & \ast & \ast\\
					0  & \ast & \ast \\
					0 & \ast & \ast 
				\end{bmatrix}.
			\end{equation*}
			\smallskip 
			We start with case (a). 
			Let $y=(0,y_2,y_3)$ and
			\begin{equation*}
				R_1(\alpha,y)=\begin{bmatrix}
					\alpha & y_2 & y_3 \\
					0 & 0& 0\\
					0 & 0& 0
				\end{bmatrix}.
			\end{equation*}
			In view of Lemma \ref{lem:techR}, if $\alpha=1$, we write $R_1(1,y)=S(y)^{-1}E_{11}S(y)$. Let
			\begin{align*}
				R_2&:=S(y)^{-1}E_{22}S(y) =E_{22}-y_2E_{12} \\
				R_3&:=S(y)^{-1}E_{33}S(y) =E_{33}-y_3E_{13}.  
			\end{align*}
			Moreover, $R_2\perp E_{33}$ and commutes with $E_{13}$, while $R_3\perp E_{22}$ and commutes with $E_{12}$. Correspondingly, the same holds for $\phi(R_2)$ and $\phi(R_3)$. Hence there exist continuous functions $f_j:\mathbb{C}\to \mathbb{C}$,   such that $\phi(R_j)=E_{jj}-f_j(y_j)E_{1j}$, $j=2,3$. Since $\phi(E_{jj})=E_{jj}$ we have $f_j(0)=0$, $j=2,3$. Let $f(y):=(0,f_2(y_2),f_3(y_3))$. By Lemma \ref{lem:techR}, for  $j=2,3$ we have
			\begin{align*}
				\phi(R_j)= S(f(y))^{-1}E_{jj}S(f(y)).
			\end{align*} 
			By Step \ref{st:0003} we have that 
			\begin{equation*}
				\phi(R_1(1,y))=S(f(y))^{-1}E_{11}S(f(y))=R_1(1,f(y)).
			\end{equation*}
			From
			\begin{equation*}
				\phi(\alpha E_{11} + y_j E_{1j})=\alpha \phi (E_{11} + y_j/\alpha E_{1j}) =\alpha E_{11} + \alpha f_j(y_j/\alpha) E_{1j},
			\end{equation*}
			the homogeneity of $\phi$, the continuity of $f_j$, and
			\begin{equation*}
				y_jE_{1j}=\lim_{\alpha\to 0} \phi(\alpha E_{11} + y_j E_{1j})
			\end{equation*}
			we get
			\begin{equation*}
				\lim_{\alpha\to 0}\alpha f_j(y_j/\alpha)=y_j,\ \qquad j=2,3.
			\end{equation*}
			Similarly, 
			\begin{align*}
				\phi(R_1(0,y)) &=\phi\left(\lim_{\alpha\to 0} R_1(\alpha,y)\right)=\alpha\phi\left(\lim_{\alpha\to 0}  R_1\left(1,y/\alpha\right)\right) =R_1\left(0, \lim_{\alpha\to 0}\alpha f\left(y/\alpha\right)\right) \\
				&=R_1(0,y).
			\end{align*}

			For arbitrary nonzero $t \in \C$ let us next consider the matrix $E_{33}-tE_{23}$, being orthogonal to $E_{11}$.  By setting $z=(0,1,t)$, $E_{33}-tE_{23}$ commutes with $R_1(0,z)=E_{12}+tE_{13}=\phi(R_1(0,z))$. So, there exist complex numbers $p_t$ and $q_t$, not both equal to zero at any nonzero $t$, such that 
			$$\phi(E_{33}-tE_{23})=\begin{bmatrix}
				0  & 0 & 0\\
				0  & -tp_z & -tq_t \\
				0 & p_t & q_t
			\end{bmatrix}.$$
			We use this observation with the relation 
			$R(1,y_2,y_3) \perp E_{33}-y_3/y_2 E_{23}$, $y_2\neq 0$, (we have chosen $t=y_3/y_2$) and thus obtain that $\left[ 1\ f_2(y_2)\  f_3(y_3) \right] \left[ 0\ -y_3/ y_2\   1\right]^t=0$ when $y_2, y_3$ are both nonzero, giving  $f_2(y_2)/y_2=f_{3}(y_3)/y_3=c$ for some nonzero complex scalar $c$. Recall that by our reduction $\phi(E_{12})=E_{12}$. By the continuity and the homogeneity of $\phi$ we have
			$$E_{12}=\phi(E_{12})=\lim_{\alpha\to 0}\phi(\alpha E_{11}+E_{12})=\lim_{\alpha\to 0}\alpha R(1,c/\alpha,0)=\lim_{\alpha\to 0} R(\alpha,c,0)=cE_{12}.$$ 
			Hence $c=1$. The cases when $y_2 = 0$ or $y_3 = 0$ follow by density, which finishes the proof of case (a).
			
			\smallskip 
			
			Let us proceed with the case (b).  We can write any such rank-one (possibly nilpotent) matrix as $ab^\ast$ where $a$, $b\in\mathbb{C}^3$ and $b \perp  e_1$. Suppose $a\notin \mathrm{span}\{e_1\}$ (otherwise, we are in the case (a)). By Step \ref{st:rk1}, $\phi(ab^\ast)=uv^\ast\in M_3$ for some nonzero vectors $u$, $v$ such that $b^\ast a=v^\ast u$. 
			Obviously, $ab^\ast$ commutes with $ e_1(a^\perp)^\ast$ for any vector $a^\perp$ orthogonal to $a$. From (a) we already know that $\phi( e_1(a^\perp)^\ast)= e_1(a^\perp)^\ast$. So, 
			\begin{equation}\label{eq:rank-one relation 3x3}
				(v^*e_1)u(a^\perp)^\ast = uv^\ast e_1(a^\perp)^\ast  = e_1(a^\perp)^\ast uv^\ast=((a^\perp)^\ast u)e_1v^\ast.
			\end{equation}
			We claim that $v \perp e_1$ and $u \perp a^\perp$. Suppose the contrary. Then, since both $u(a^\perp)^\ast$ and $e_1v^\ast$ are nonzero, \eqref{eq:rank-one relation 3x3} shows that both $v^*e_1$ and $(a^\perp)^\ast u$ are zero or both nonzero. If they are nonzero, we have  $u\in \mathrm{span}\{e_1\}$, which conflicts injectivity. Therefore, $v \perp e_1$ and $u \perp a^\perp$ 
			and so, $u\in ( \{a\}^\perp )^\perp=\mathbb{C}a$. Without loss of any generality, we can assume $u=a$, and the scalar factor can be absorbed in $v$. So far we have obtained that $\phi(ab^\ast)=av^\ast\in \mathcal{A}$ for some vector $v\perp e_1$ which depends on $a$ and $b$. On the other hand, $ab^\ast$ commutes with $(b^\perp)c^\ast$,  where $b^\perp \perp b$ and $c\in \{e_1,a\}^\perp$ (it is possible that $c$ can be chosen only in the way that $(b^\perp)c^\ast$ is nilpotent and so, we can use only commutativity preserving). From the previous argument, we know that $\phi(b^\perp c^\ast)=b^\perp w^\ast$ for some $w\ne 0$. Then $\phi(ab^\ast)=av^\ast \leftrightarrow b^\perp w^\ast=\phi(b^\perp c^\ast)$. This gives 
			$$(v^\ast b^\perp)aw^\ast = av^* b^\perp w^* = b^\perp w^* av^* =(w^\ast a)b^\perp v^\ast.$$ 
			At this point assume that $ab^\ast$ is not nilpotent, i.e.\ $b$ is not orthogonal to $a$. Then $b^\perp$ (defined up to a scalar factor) and $a$ are linearly independent, hence $v^\ast b^\perp=0$. This implies that $v=\beta b$ for some scalar $\beta\neq 0$. From the spectrum preserving property, comparing the traces, we get that $\beta =1$, and we are done.

			\bigskip

													\begin{stepa}[\textbf{Conclusion of the base step}]\label{step:finish 3x3}
														$\phi(A) = A$ for all matrices $A \in \mathcal{A}$.
													\end{stepa}
													By density, it suffices to prove that $\phi$ is the identity map on the set of matrices in $\mathcal{A}$ with $3$ distinct eigenvalues. By Lemma \ref{lem:diagonalizes within A} for every matrix $A$ of this kind there exists a matrix  $S \in \mathcal{A}^{\times}$ and a diagonal matrix $D$ such that $A=SDS^{-1}$. Applying Step 1 for the map $X\mapsto \phi(SXS^{-1})$ we get that there exists $T \in M_3^{\times}$ such that
													$$\phi(SDS^{-1}) = TDT^{-1}, \quad \text{ for all }D\in\mathcal{D}_3.$$
													By Step \ref{step:identity on rank-one matrices 3x3} we also have
													$$SE_{jj}S^{-1} = \phi(SE_{jj}S^{-1}) =TE_{jj}T^{-1}, \qquad 1 \le j \le 3,$$
													so by linearity  $TDT^{-1} = SDS^{-1}$ and consequently,
													$$\phi(SDS^{-1}) = TDT^{-1} = SDS^{-1}$$
													for all diagonal matrices $D$. We conclude that $\phi$ is the identity map.

													\begin{step}[\textbf{Induction step}]
														By way of induction, suppose that $n \ge 4$ and that Theorem \ref{thm:main result} holds for $n-1$.
													\end{step}
													By Step 1, we can assume without loss of any generality, that $\phi(D)=D$ for any diagonal matrix $D$ and, in particular, $\phi(E_{ii})=E_{ii}$ for all $1\le i \le n$.  
													\begin{stepa}\label{step:phi na ortogonalu}
														\textit{There exists a diagonal matrix $T \in \mathcal{D}_n$ such that $\phi(A)=TAT^{-1}$ for all $ A\in \mathcal{A} \cap E_{ss}^\perp$, $1 \le s \le n$ or,  $\phi(A)=TA^tT^{-1}$ for all  $ A\in \mathcal{A} \cap E_{ss}^\perp$, $1 \le s \le n$.}
													\end{stepa}	
													\begin{proof} By the continuity of $\phi$ it suffices to consider only diagonalizable matrices. By Step \ref{st:hom0}, for any diagonalizable matrix $A\in \mathcal{A}\cap E_{ss}^\perp$ we have that $\phi(A) \in E_{ss}^\perp$. In a natural way $E_{ss}^\perp$ is isomorphic to a block upper-triangular subalgebra of $M_{n-1}$ and hence, the induction hypothesis can be applied. Having this in mind,   for each $1 \le s \le n$ there exists an invertible matrix $T_s \in M_n^\times$ commuting with $E_{ss}$ such that the map $\phi$ satisfies
														\begin{equation}\label{eq:s}
															\phi(A) = T_s A T_s^{-1},\ \ \   A\in \mathcal{A}\cap E_{ss}^\perp,  
														\end{equation}	
														or
														\begin{equation*}
															\phi(A) = T_s A^t T_s^{-1},\ \ \   A\in \mathcal{A}\cap E_{ss}^\perp.
														\end{equation*}
														The matrix $T_s$ is diagonal by our assumption preceding this step. 
														We can set the $s$-th diagonal entry of $T_s$ to be equal to $1$ for every $s$. We may assume that $\phi$ is of the form \eqref{eq:s} when $s=1$; otherwise we replace $\phi$ by the map $\phi(\cdot)^t$.
														Replacing $\phi$ by the map $T_1^{-1}\phi(\cdot)T_1$ we can assume that $\phi$ acts as identity on $E_{11}^\perp \cap \mathcal{A}$.  
														\smallskip
														
														Let $s\ne 1$. Applying that for every $A\in E_{11}^\perp \cap E_{ss}^\perp\cap \mathcal{A}$ holds $\phi(A)=T_{1}AT_1^{-1} = A$ and that $E_{11}^\perp \cap E_{ss}^\perp\cap \mathcal{A}$ contains at least $\mathrm{span}\{E_{jj},E_{kk},E_{jk}\}$ for some $1<j<k\leq n$  as $n\geq 4$, we conclude that  $\phi$ is of the form \eqref{eq:s} on $E_{ss}^\perp \cap  \mathcal{A}$ since otherwise the restriction of $\phi$ to $E_{11}^\perp \cap E_{ss}^\perp\cap \mathcal{A}$ would be multiplicative and antimultiplicative simultaneously. Therefore, as every $E_{ij}\in \mathcal{A}$ belongs to  $E_{ss}^\perp$ for some $s\notin \{ i,j\}$, $\phi(E_{ij})=c_{ij}E_{ij}$, $c_{ij}\in \C^\times$, $c_{ij}=1$ for all $i,j$ such that $E_{ij}\in E_{11}^\perp \cap \mathcal{A}$. Furthermore, replacing $\phi$ by $S^{-1}\phi(\cdot)S$, where $S=\diag(c_{12},I_{n-1})$, which does not affect matrices in $E_{11}^\perp \cap \mathcal{A}$, we assume that $\phi(E_{12})=E_{12}$. Note that $\phi$ is multiplicative on each $E_{ss}^\perp \cap \mathcal{A}$. So, we have that for any $2<j\le n$ 
														$$c_{1j}E_{1j}=\phi(E_{1j})=\phi(E_{12}E_{2j})=\phi(E_{12})\phi(E_{2j})=E_{12}E_{2j}=E_{1j}$$
														which gives $c_{1j}=1$ for all $2\le j \le n$. For any $1<j$ such that $E_{j1}\in \mathcal{A}$ we also  have
														$$E_{jj}=\phi(E_{jj})=\phi(E_{1j}E_{j1})=\phi(E_{1j})\phi(E_{j1})=c_{j1}E_{1j}E_{j1}=c_{j1}E_{jj}$$
														providing $c_{j1}=1$. 
														As now $\phi(E_{ij})=E_{ij}$ for all $E_{ij}\in \mathcal{A}$, we conclude that $T_s=I$ for all $s=1,2,\dots,n$.  
														
														\smallskip
														
													\end{proof}
													\smallskip
													
													By Step \ref{step:phi na ortogonalu} we assume that $\phi(A)=A$ for every matrix $A\in E_{ss}^\perp \cap \mathcal{A}$, $1\le s\le n$.
													
													\begin{stepa}\label{st:R} $\phi(R)=R$ \textit{for every rank-one matrix} $R\in \mathcal{A}$.
													\end{stepa}
													\begin{proof}	
														Recall that at the beginning of the proof of Theorem \ref{thm:main result} we assumed $r \ge 2$ (i.e.\ $\mathcal{A} \subsetneq M_n$).
														It suffices to consider only idempotent matrices $R$ due to the continuity and the homogeneity of $\phi$. \smallskip
														
														We first show that $\phi(R)=R$ for every matrix $R$ supported in the first row.  By Lemma \ref{lem:techR}, such a matrix can be written as  
														$$R:= e_1\left[1\  y_2\  \dots\  y_n \right]= S_1(y)^{-1}E_{11}S_1(y),$$
														where  $y=\sum_{i=2}^{n} y_i e_i$.  We observe that for every $i\notin \{1,n\} $ the matrix 
														$$R_i:= S_1(y)^{-1}E_{ii}S_1(y)=E_{ii}-y_iE_{1i}$$
														is orthogonal to $ E_{nn}$  and $R_n\perp E_{22}$.  Therefore,  $\phi(R_i)=R_i$ for every $i=2,\dots, n$ and thus, by Step \ref{st:0003}, $\phi(R)=R$. 
														\smallskip

														For the matrices being supported only in the last column, we apply that the map $\widetilde{\phi}:\mathcal{A}^\odot \to M_n$, $\widetilde{\phi}(B)=\phi(B^\odot)^\odot$, $B\in \mathcal{A}^\odot$, is injective, continuous, preserves spectrum and commutativity and moreover, $\widetilde{\phi}(E_{ij})=E_{ij}$ for all $E_{ij}\in \mathcal{A}^\odot$. Consequently, $\widetilde{\phi}$ fixes all idempotent matrices supported only in the first row. Then for every idempotent matrix $C$ supported in the last column, $\phi(C)=\widetilde{\phi}(C^\odot)^\odot = C$.
														\smallskip 

														Let us further suppose that $R$ is neither supported in the first row nor in the last column. Hence, let $a,b\in \mathbb{C}^n$ be such that $R=ab^\ast\in \mathcal{A}$. By our assumption on $R$ we have $e_n^\ast a=b^\ast e_1 =0$ and $b^*a =1$. Applying Step \ref{st:rk1}, write $\phi(ab^\ast) = xy^\ast$ for some $x,y\in\mathbb{C}^n$  with $y^\ast x =1$.  Since $ab^* \perp E_{1n}$, we have $\phi(ab^*) \leftrightarrow E_{1n}$, which implies
														$$(y^*e_1) xe_n^* = (xy^*)e_1e_n^* = (e_1e_n^*)xy^* = (e_n^*x)e_1 y^*.$$ 
														If $y^*e_1, e_n^*x \ne 0$, we obtain $x = x_1e_1$ and $y = y_ne_n$ which would imply that $\phi(R) = xy^*$ is a nonzero multiple of $E_{1n}$. By the injectivity and homogeneity of $\phi$, from $\phi(E_{1n}) = E_{1n}$ it would follow that $R$ is also a nonzero multiple of $E_{1n}$, but this contradicts the fact that $R$ is an idempotent. Therefore, we conclude $y^*e_1 =  e_n^*x = 0$. Choose any nonzero $a^\perp\in \{a\}^\perp$ and $b^\perp\in \{b\}^\perp$. By relations $e_1(a^\perp)^\ast\perp  ab^\ast$ and  $ b^\perp e_n^\ast \perp ab^\ast $ we get $e_1(a^\perp)^\ast=\phi(e_1(a^\perp)^\ast)\perp  xy^\ast$ and  $ b^\perp e_n^\ast =\phi(b^\perp e_n^\ast)\perp xy^\ast $. These relations  give
														$$ ((a^\perp)^\ast x) e_1 y^\ast= (e_1 (a^\perp)^*) xy^*=(xy^* )e_1 (a^\perp)^* = (y^\ast e_1) x(a^\perp)^\ast = 0$$
														and
														$$0 = (e_n^\ast x)b^\perp y^\ast= (b^\perp e_n^*) xy^* = (xy^* )b^\perp e_n^* = (y^\ast b^\perp)x e_n^\ast.$$
														From $(a^\perp)^\ast x=0=y^\ast b^\perp$ being true for every $a^\perp\perp a$ and every $b^\perp\perp b$ we see that  $x=\alpha a$, $y=\beta b$ for some nonzero scalars $\alpha$, $\beta$ and hence, $\phi(ab^\ast)= (\alpha \beta) ab^\ast$. Comparing the traces of $R$ and $\phi(R)$ we see that $\alpha\beta=1$ and thus $\phi(R)=R$.
													\end{proof}
													\bigskip

																\begin{step}[\textbf{Conclusion of the inductive step}] $\phi(A) = A$ \textit{for all matrices} $A \in \mathcal{A}$.   \label{st:6}
																\end{step}
																
																\begin{proof}	
																	The verification of this step is essentially the same as in Step \ref{step:finish 3x3} for the $3\times 3$ case.
																\end{proof}
															\end{proof}

															\begin{theorem}\label{between two block upper-triangular algebras}
																Let $\mathcal{A}, \mathcal{B} \subseteq M_n$ be two block upper-triangular algebras and let $\phi : \mathcal{A} \to \mathcal{B}$ be a continuous injective spectrum and commutativity preserving map. Then one of the following is true:
																\begin{itemize}
																	\item $\mathcal{A} \subseteq \mathcal{B}$ and there exists $T \in \mathcal{B}^{\times}$ such that $\phi(X) = TXT^{-1}$ for all $X \in \mathcal{A}$.
																	\item $\mathcal{A}^{\symm} \subseteq \mathcal{B}$ and there exists $T \in \mathcal{B}^{\times}$ such that $\phi(X) = TX^{\symm}T^{-1}$ for all $X \in \mathcal{A}$.
																\end{itemize}
															\end{theorem}
															\begin{proof}
																Theorem \ref{thm:main result} implies that $\phi$ is a Jordan embedding. The rest of the result follows from Corollary \ref{cor:Jordan monomorphism unified}.
															\end{proof}
															
															When we further assume that $\mathcal{B} = \mathcal{A}$, so that $\phi : \mathcal{A} \to \mathcal{A}$, we can relax the spectrum preserving assumption to spectrum shrinking ($\sigma(\phi(X)) \subseteq \sigma(X)$ for all $X \in \mathcal{A}$). More precisely, we obtain the following result (similarly as in \cite{Tomasevic}, the proof relies on the invariance of domain theorem).
															
															\begin{corollary}\label{cor:invdomain}
																Let $\mathcal{A} \subseteq M_n$ be a block upper-triangular algebra and let $\phi : \mathcal{A} \to \mathcal{A}$ be a continuous injective commutativity preserving spectrum shrinking map. Then one of the following is true:
																\begin{itemize}
																	\item There exists $T \in \mathcal{A}^{\times}$ such that $\phi(X) = TXT^{-1}$ for all $X \in \mathcal{A}$.
																	\item $\mathcal{A}^{\symm} = \mathcal{A}$ and there exists $T \in \mathcal{A}^{\times}$ such that $\phi(X) = TX^{\symm}T^{-1}$ for all $X \in \mathcal{A}$.
																\end{itemize}
															\end{corollary}
															
															\begin{proof} The result follows immediately from Theorem \ref{between two block upper-triangular algebras} once we prove $\phi$ is spectrum preserving. By the invariance of domain theorem, the image $\mathcal{R} = \phi(\mathcal{A})$ is an open set in $\mathcal{A}$ and $\phi$ is a homeomorphism onto $\mathcal{R}$. Then $\phi^{-1}:\mathcal{R}\to \mathcal{A}$ enlarges the spectrum: $\sigma(A)\subseteq \sigma (\phi^{-1}(A))$. By Lemma \ref{lem:char} the characteristic polynomial is being preserved and hence, $\phi$ preserves the spectrum.
															\end{proof}

															\section{Counterexamples}\label{sec:Counterexamples}
															
															We show the optimality of Theorem \ref{thm:main result} via counterexamples. In short, all assumptions except injectivity are indispensable for all block upper-triangular algebras except $M_n$, while injectivity is superfluous in the $M_n$ case and necessary in all other cases. The necessity of $n \ge 3$ is shown by examples from \cite[Example 7]{PetekSemrl} and \cite[Theorem 4]{Petek}. Therefore we assume $n \ge 3$ and let $\mathcal{A}=\mathcal{A}_{k_1,\dots,k_r}\subseteq M_n$ be an arbitrary block upper-triangular algebra not equal to $M_n$.
															
															\begin{example}[Injectivity, continuity and commutativity preserving are not enough]
																Let $D$ be the open unit disk in $\C$. Define a map $g : D \to D$ by
																$$
																g(z) = \frac{1-3z}{3-z}.
																$$
																It is easy to check that $g$ is a holomorphic bijection (actually, it is an involution). Consider the map $\phi : \mathcal{A} \to M_n$ given by
																$$\phi(X) = g\left(\frac{X}{1+\norm{X}}\right),$$
																where $\norm{\cdot}$ denotes the spectral norm. The map $\phi$ is well-defined, as for each $X \in \mathcal{A}$ the matrix $\frac{X}{1+\norm{X}}$ has norm $< 1$ and hence its spectrum is contained in $D$, at which point we can apply $g$ using the holomorphic functional calculus. Using the properties of the holomorphic functional calculus, we conclude that $\phi$ is continuous and preserves commutativity. Moreover, since the map $X \mapsto \frac{X}{1+\norm{X}}$ is injective, via the application of $g^{-1}=g$ we conclude that $\phi$ is injective. However, $\phi$ is clearly not linear as
																$$\phi(0) = g(0)= \frac13 I.$$
															\end{example}
															
															\begin{example}[Commutativity preserving is necessary]
																Consider the map $f : M_n \to M_n^{\times}$ given by
																$$f(X)=\diag(e^{\det X}, 1, \ldots, 1).$$
																Then the map $\phi : \mathcal{A} \to M_n$ given by
																$$\phi(X) = f(X)Xf(X)^{-1}$$
																is clearly continuous and preserves the spectrum. Moreover, since $f$ is a similarity invariant and $\phi(X)$ is similar to $X$, we can see that $\phi$ is bijective with the inverse $Y \mapsto f(Y)^{-1}Yf(Y)$. However, $\phi$ is not linear:
																$$\phi(I+E_{12}) = I+eE_{12} \ne I+E_{12} = \phi(I)+\phi(E_{12}).$$
															\end{example}
															
															\begin{example}[Continuity is necessary]
																As in \cite{PetekSemrl}, consider the map $\phi : \mathcal{A} \to M_n$ given by $$\phi(X) = \begin{cases} \diag(\lambda_2,\lambda_1, \ldots, \lambda_n), \quad &\text{ if }X = \diag(\lambda_1, \ldots, \lambda_n) \text{ and all $\lambda_i$ are distinct,}\\
																	X, &\text{ otherwise}
																\end{cases}$$
																which is bijective, spectrum and commutativity preserving but clearly not continuous. 
															\end{example}
															
															\begin{example}[Injectivity is necessary] Consider the map $\phi : \mathcal{A} \to M_n$ given by 
																$$\phi\left(\begin{bmatrix} X_{11} & X_{12} & \cdots & X_{1n} \\ 0 & X_{22} & \cdots & X_{2n} \\
																	\vdots & \vdots & \ddots & \vdots \\
																	0 & 0 & \cdots & X_{rr} \end{bmatrix}\right)=\begin{bmatrix} X_{11} & 0 & \cdots & 0 \\
																	0 & X_{22} & \cdots & 0 \\
																	\vdots & \vdots & \ddots & \vdots \\
																	0 & 0 & \cdots & X_{rr} \end{bmatrix}.$$
																Then $\phi$ is clearly a unital Jordan homomorphism (and hence satisfies all assumptions of Theorem \ref{thm:main result}), but is not injective.
															\end{example}
															
															\section*{Funding}
															The second author is supported by the Slovenian Research and Innovation  Agency (core research program P1-0306).

														\end{document}